\documentclass[11pt]{article}

\usepackage{enumerate}
\usepackage{amsmath}
\usepackage{amssymb,latexsym}
\usepackage{amsthm}
\usepackage{color}
\usepackage{graphicx}
\usepackage{amscd}
\usepackage{hyperref}
\usepackage{multicol}
\usepackage{textgreek}

\title{Scattered linear sets in a finite projective line and translation planes}
\author{Valentina Casarino, Giovanni Longobardi, Corrado Zanella}
\date{}

\newcommand{\ptra}{{\mathcal A}}
\newcommand{\cA}{{\mathcal A}}
\newcommand{\cB}{{\mathcal B}}

\newcommand{\cD}{{\mathcal D}}

\newcommand{\F}{{\mathbb F}}

\newcommand{\Fq}{\F_q}
\newcommand{\Fp}{\F_p}

\newcommand{\Fqt}{\F_{q^t}}

\newcommand{\la}{\langle}
\newcommand{\ra}{\rangle}

\newcommand{\Mod}[1]{\ (\mathrm{mod}\ #1)}

\newtheorem{theorem}{Theorem}[section]

\newtheorem{corollary}[theorem]{Corollary}
\newtheorem{proposition}[theorem]{Proposition}

\DeclareMathOperator{\PG}{{PG}}

\DeclareMathOperator{\GL}{{GL}}
\DeclareMathOperator{\PGL}{{PGL}}

\DeclareMathOperator{\GaL}{\Gamma L}
\DeclareMathOperator{\Gal}{Gal}

\DeclareMathOperator{\N}{N}

\DeclareMathOperator{\id}{id}

\theoremstyle{definition}%% Hans
\newtheorem{definition}[theorem]{Definition}

\newtheorem{remark}[theorem]{Remark}

\begin{document}

\maketitle
\begin{abstract}
In \cite{LuPo01}, Lunardon and Polverino construct a translation plane starting from a
scattered linear set of pseudoregulus type in $\PG(1,q^t)$.
In this paper a similar construction of a translation plane  $\ptra_f$
obtained from any scattered linearized polynomial $f(x)$ in $\Fqt[x]$ is described and investigated.
A class of quasifields giving rise to such planes is defined.
Denote by $U_f$ the $\Fq$-subspace of $\Fqt^2$ associated with $f(x)$.
If $f(x)$ and $f'(x)$ are scattered,
then $\ptra_f$ and $\ptra_{f'}$ are isomorphic if and only if
$U_f$ and $U_{f'}$ belong to the same orbit under the action of $\GaL(2,q^t)$.
This gives rise to the same number of distinct translation planes as the number of inequivalent 
scattered linearized polynomials.
In particular, for any scattered linear set $L$ of maximum rank in $\PG(1,q^t)$ there are 
$c_\Gamma(L)$ 
pairwise  non-isomorphic translation planes, where $c_\Gamma(L)$ 
denotes the $\GaL$-class of $L$,
as defined in \cite{CsMaPo18} by Csajb\'ok, Marino and Polverino.
A result by Jha and Johnson \cite{JhJo08} allows to describe the automorphism groups of
the planes obtained from the linear sets not of pseudoregulus type defined in \cite{LuPo01}. 
\end{abstract}



\section{Introduction}
The purpose of this work is to study
translation planes  that arise  from   a  scattered linear set
of maximum rank in a finite projective line by replacing a related hyper-regulus.
Early investigations of this kind, concerning
the scattered linear sets of pseudoregulus type in $\PG(1,q^t)$, can be found
in    \cite{LuPo01}.
The focus in that paper was on the theory of blocking sets,
%That paper  was mainly motivated by  blocking sets  theory, 
 but it  is worth noticing that  the planes constructed in   \cite{LuPo01} %%may be rephrased in terms of linear sets.
only depend on linear sets in the projective line.

Linear sets  
generalize the notion of subgeometry of a projective space and 
have ubiquitous applications in finite
geometry,  being connected with many different  %mathematical 
objects, such as blocking sets, two-intersection sets, complete caps, translation
spreads of the Cayley Generalized Hexagon, translation ovoids of polar
spaces, semifield flocks, finite semifields and rank metric codes 
(see e.g. \cite{LaVV15,Po10,Sh16}).
The construction carried out in this article is also motivated by some results in \cite{JhJo08}, where  
the authors  introduce a new class of hyper-reguli in a Desarguesian spread, and then 
investigate the translation planes obtained by the replacement of such hyper-reguli.
They prove that such planes are neither Andr\'e nor generalized Andr\'e planes.

What follows  is %a plan of the proof of Theorem 1.1, together with 
a description both of the main results and 
of the structure of this paper.
%%%The structure of the paper is as follows.
In Section \ref{Notations}  some terminology and notation are introduced.
%, leading, in particular, to a precise definition of an  Andr\'e $q$-plane. 
In Section \ref{s:spaq}
a quasifield $\mathcal Q_f$ and  a  translation plane $\mathcal A_f$ are suitably associated 
with any scattered
$\Fq$-linearized polynomial $f(x)\in\Fqt[x]$, $q>2$.
Section \ref{S.equiv} is devoted to the question of 
isomorphism between translation planes corresponding  to distinct scattered polynomials;
a characterization of isomorphic translation planes is provided
(Theorems \ref{t:main} and  \ref{th:pseudoreg}). 
More precisely, given a scattered linearized polynomial $f(x)\in\Fqt[x]$,
denote by $U_f$ and $L_f$ the $\Fq$-subspace of $\Fqt^2$ and the $\Fq$-linear set in
$\PG(1,q^t)$ associated with $f(x)$, respectively.
Let $f(x), f'(x)\in\Fqt[x]$ be  scattered linearized polynomials;
then $\ptra_f$ and $\ptra_{f'}$ are isomorphic if and only if
$f(x)$ and $f'(x)$ are equivalent, that is,
$U_f$ and $U_{f'}$ belong to the same orbit under the action of $\GaL(2,q^t)$.

The number of known scattered linearized polynomials is
constantly growing and also includes classes of infinitely many
extensions of a field $\Fq$: see e.g.\ \cite{BaZaZu20,LoMaTrZh21,LoZa21}
and the references therein. By virtue of the main
result of this paper (Theorem \ref{t:main}), this leads to as many
distinct translation planes.

As is known \cite{CsZa16}, two $\Fq$-subspaces of $\Fqt^2$ defining the same scattered linear
set of maximum rank in $\PG(1,q^t)$ not necessarily are in the same orbit under the action of 
$\GaL(2,q^t)$.
This led to the notion of $\GaL$-class of a linear set $L$ 
\cite{CsMaPo18} as the maximum number $c_\Gamma(L)$
of $\Fq$-subspaces defining $L$ and belonging to distinct orbits under 
the action of $\GaL(2,q^t)$.
As a consequence of Theorem \ref{t:main}, any scattered linear set $L$ of maximum rank in $\PG(1,q^t)$ 
gives rise to $c_\Gamma(L)$
pairwise  non-isomorphic translation planes.

In Section \ref{s:LP} the discussion is focused on the number and the automorphism
group of the
translation planes associated with  Lunardon-Polverino polynomials, which are described
 by means of the results in \cite{JhJo08}. 
%In the final section some open questions are posed.

\section{Notation and preliminaries}\label{Notations}
Let $q$ be a power of a prime $p$, and $r,t\in\mathbb N$, $r>0$, $t>1$.
Let $U$ be an $r$-dimensional $\Fq$-subspace of the $(dt)$-dimensional vector space
$V(\Fqt^d,\Fq)$.
For a set $S$ of field elements (or vectors), we denote by $S^*$ the set of non-zero elements 
(non-zero vectors) of $S$.
The following subset of $\PG(d-1,q^t)=\PG(\Fqt^d,\Fqt)$
\[
L_U=\{\la v\ra_{\Fqt}\colon v\in U^*\}
\]
%(\footnote{In this note, $S^*=S\setminus\{0\}$ for any set containing a zero element.})
is called \emph{$\Fq$-linear set} (or just \emph{linear set}) of \emph{rank $r$}.
The linear set $L_U$ is \emph{scattered} if it has the maximum possible size related to the
given $q$ and $r$, that is, $\#L_U=(q^r-1)/(q-1)$.
Equivalently, $L_U$ is scattered if and only if
\begin{equation}\label{e:scattered}
  \dim_{\Fq}\left(U\cap\la v\ra_{\Fqt}\right)\le1\quad\mbox{ for any }v\in\Fqt^d.
\end{equation}
Any $\Fq$-subspace $U$ satisfying \eqref{e:scattered} is a \emph{scattered $\Fq$-subspace}.

Clearly, any $\Fq$-linear set in $\PG(d-1,q^t)$ of rank greater than $t(d-1)$ coincides
with $\PG(d-1,q^t)$.
The scattered linear sets of rank as large as
possible are called of \emph{maximum rank}. By \cite{BlLa00}, any scattered $\Fq$-linear set in
$\PG(d - 1, q^t )$ has rank at most $td/2$; for $td$ even the bound is sharp~\cite{CsMaPoZu17}.

In this note only $\Fq$-linear sets of maximum rank in $\PG(1,q^t)$ are dealt with.
They are associated with linearized polynomials.
An \emph{$\Fq$-linearized polynomial} in $\Fqt[x]$ is of type $f(x)=\sum_{i=0}^ka_ix^{q^i}$ ($k\in\mathbb N$).
If $a_k\neq0$, then $k$ is the \emph{$q$-degree} of $f(x)$.
It is well-known that the $\Fq$-linearized polynomials of $q$-degree less than $t$ in $\Fqt[x]$
are in one-to-one correspondence with the endomorphisms of the vector space
$V(\Fqt,\Fq)$. 
An $\Fq$-linearized polynomial $f(x)\in\Fqt[x]$ is called \emph{scattered} if for any $y,z\in\Fqt$
the condition $zf(y)-yf(z)=0$ implies that $y$ and $z$ are $\Fq$-linearly dependent.

Let $f(x)\in\Fqt[x]$ be an $\Fq$-linearized polynomial and define 
\[U_f=\{(x,f(x))\colon x\in\Fqt\},\]
and $L_f=L_{U_f}$.
Such $L_f$ is an $\Fq$-linear set of maximum rank of $\PG(1,q^t)$, and is scattered if and only if
$f(x)$ is.
Two $\Fq$-linearized polynomials $f(x)$ and $f'(x)$ in $\Fqt[x]$ are said to be \emph{equivalent}
when a $\kappa\in\GaL(2,q^t)$ exists such that $U_f^\kappa=U_f'$ \cite{CsMaPo18}.

%Since $\PGL(2,q^t)$ acts $3$-transitively on $\PG(1,q^t)$, if $q>2$ 
Since $\PGL(2, q^t )$ acts $3$-transitively on $\PG(1, q^t )$ and
since the size of an $\Fq$-linear set of $\PG(1, q^t)$ is at most $(q^t - 1)/(q - 1)$, if
$q > 2$, then
any linear set of maximum rank is projectively equivalent to an $L_U$ such that
\[
\left\{\la(1,0)\ra_{\Fqt},\la(0,1)\ra_{\Fqt},\la(1,1)\ra_{\Fqt}\right\}\cap L_U=\emptyset.
\]
Therefore,
a linearized polynomial $f(x)$ exists such that $U=U_f$ and $\ker f=\{0\}=\ker(f-\id)$.

By abuse of notation, $L_f$ will also denote the set $\{f(x)/x\colon x\in\Fqt^*\}$ of the
nonhomogeneous projective coordinates of the points belonging to the set $L_f$.
According to the convention above for $f(x)$, $0,1\notin L_f$.

%\medskip

A (right) \emph{quasifield} is an ordered triple $(Q,+,\circ)$, where $Q$ is a set,
$(Q,+)$ is an abelian group, $(Q^*,\circ)$ is a loop, $(x+y)\circ m=x\circ m+y\circ m$
for any $x,y,m\in Q$, and for any $a,b,c\in Q$ with $a\neq b$, 
the following equation in the
unknown $x\in Q$ has a unique solution: $x\circ a=x\circ b+c$.
The \emph{kernel} of the quasifield  $(Q,+,\circ)$ is
\begin{gather*}
  K(Q)=\{k\in Q\colon k\circ(x+y)=k\circ x+k\circ y\mbox{ and }
  k\circ(x\circ y)=(k\circ x)\circ y\\ \mbox{ for any }x,y\in Q\}.
\end{gather*}
The kernel of $Q$ is a skewfield and $Q$ is a left vector space over $K(Q)$ 
\cite{An54}\cite[Chapter 1]{Kn95}.
A  quasifield satisfying the right distributive property is a \emph{semifield}; an associative quasifield is a
\emph{nearfield}.
From now on, any quasifield is assumed to be finite.
Therefore, $K(Q)$ is a finite field.

%\medskip

A \emph{partial planar spread} of a vector space $\mathbb V$ is a collection $\cB$ of at least
three subspaces
such that $\mathbb V=V\oplus V'$ for any $V,V'\in\cB$ with $V\neq V'$.
If $\cB$ is a partial planar spread of $\mathbb V$ such that
for any $v\in\mathbb V$ there exists a $V\in\cB$ such that $v\in V$, then $\cB$ is a
\emph{planar spread} of $\mathbb V$.
The \emph{Desarguesian planar spread} of the $(2t)$-dimensional $\Fq$-vector space $\Fqt^2$ is
\[
  \cD=\left\{\la v\ra_{\Fqt}\colon v\in\left(\Fqt^2\right)^*\right\}.
\]

Let $(Q,+,\circ)$ be a finite quasifield.
Define $V_\infty=\{0\}\times Q$, and
\[
  V_m=\{(x,x\circ m)\colon x\in Q\}\mbox{ for }m\in Q.
\]
Then $\cB(Q)=\{V_m\colon m\in Q\cup\{\infty\}\}$ is a planar spread of the $K(Q)$-vector space
$Q^2$.
Conversely, for any planar spread $\cB$ a quasifield $Q$ exists such that $\cB=\cB(Q)$.

Let $\cB$ be a planar spread of a vector space $\mathbb V$ over $\Fqt$.
Other planar spreads can be constructed by using the following technique of \emph{net replacement}.
If $\cB'$ and $\cB''$ are distinct partial planar spreads of $\mathbb V$ such that
$\cB'\subseteq\cB$, and
\[
  \bigcup_{V'\in\cB'}V'=\bigcup_{V''\in\cB''}V'',
\]
then $(\cB\setminus\cB')\cup\cB''$ is a planar spread of $\mathbb V$.
%Such $\cB'$ is called \emph{replaceable (or derivable) net}.
In the special case where $\dim_{\Fq}\mathbb V=2t$ and $\#\cB'=\#\cB''=(q^t-1)/(q-1)$,
the partial spreads $\cB'$ and $\cB''$ are called \emph{hyper-reguli} \cite[Chapter 19]{BiJhJo07}.
Such hyper-reguli are also called \emph{replacement set} of each other.
In some of the literature the definition requires that $\dim_{\Fq}(V'\cap V'')=1$ for any
$V'\in\cB'$ and $V''\in\cB''$, e.g.~\cite{JhJo08}.

In \cite{LuPo01}, Lunardon and Polverino show that $\{\la(x,x^q)\ra_{\Fqt}\colon x\in\Fqt^*\}$
is a hyper-regulus of $\Fqt^2$, contained in $\cD$, leading to a net replacement.
Actually the same arguments as in \cite[Section 3]{LuPo01} lead to a hyper-regulus 
starting from
any scattered $\Fq$-linear set of maximum rank in $\PG(1,q^t)$.
This fact is folklore and addressed in Section \ref{s:spaq}. 

The \emph{translation plane $\ptra(\cB)$ associated with the planar spread} $\cB$ of $\mathbb V$
is the plane whose points are the elements of $\mathbb V$  and whose
lines are the cosets of the elements of $\cB$ in the group $(\mathbb V,+)$.
Given a quasifield $Q$, the \emph{translation plane $\ptra(Q)$ associated with $Q$}  
is the translation plane $\ptra(\cB(Q))$ associated with $\cB(Q)$.
In other words, the lines 
are represented by the equations of type $x=b$ and $y=x\circ m+b$ ($m,b\in Q$).

If $\ptra(Q)$ and $\ptra(Q')$ ($Q$ and $Q'$ two quasifields) are isomorphic
translation planes, then $K(Q)$ and $K(Q')$ are isomorphic fields.
Hence the kernel of $Q$ is a geometric invariant of the plane $\ptra(Q)$,
also called the \emph{kernel of the translation plane $\ptra(Q)$}.
See \cite{Kn95} for a purely geometric definition.

Let $\N_{q^t/q}(m)=m^{(q^t-1)/(q-1)}$ denote the \emph{norm} of $m\in\Fqt$ over $\Fq$. 
An \emph{Andr\'e $q$-plane} is associated to the planar spread of $\Fqt^2$ obtained from
$\cD$ by replacing each partial spread of type
\[
  \cB'_{\xi}=\{\la(1,m)\ra_{\Fqt}\colon \N_{q^t/q}(m)=\xi\}\qquad(\xi\in\Fq^*)
\]
with
\[
  \cB''_{\mu(\xi)}=\left\{\{(x,x^{\mu(\xi)}m)\colon x\in\Fqt\}\colon \N_{q^t/q}(m)=\xi\right\},
\]
where $\mu:\Fq^*\rightarrow\Gal(\Fqt/\Fq)$ is any map.
Actually, $\mu(\xi)=\id$ means no replacement for that value of $\xi$.
Any $\cB'_{\xi}$ is called \emph{Andr\'e $q$-net}.
If for $\xi\in\Fq^*$, the map $\mu(\xi)$ is $x\mapsto x^{q^s}$,
the partial spread $\cB''_{\mu(\xi)}$ is an \emph{Andr\'e $q^s$-replacement}
\cite[Definition 16.1]{BiJhJo07}.

\section{Scattered polynomials and quasifields}\label{s:spaq}

In this section $f(x)\in\Fqt[x]$ is a scattered $\Fq$-polynomial, $q>2$, satisfying $0,1\notin L_f$.
\begin{proposition}\label{p:quasicorpo}
Let $q>2$.
Let $Q_f=\Fqt$ be endowed with the sum of $\Fqt$, and define
\begin{equation}\label{e:quasicorpo}
x\circ m=\begin{cases}
xm&\mbox{ if }m\notin L_f,\\
h^{-1}f(hx)&\mbox{ if }m\in L_f\mbox{ and }f(h)-mh=0,\ h\neq0
\end{cases}
\end{equation}
for any $x,m\in Q_f$.
Then $(Q_f,+,\circ)$ is a quasifield, and $K(Q_f)=\Fq$.
\end{proposition}
\begin{proof}
To verify the quasifield axioms is a routine proof.

Next, note that $x\circ m=xm$ for any $x\in\Fq$ and $m\in Q_f$.
Furthermore, the map $\rho_m: Q_f\rightarrow Q_f$ defined by
\begin{equation}\label{e:rhom}
\rho_m(x)=x\circ m
\end{equation}
is $\Fq$-linear for any $m\in Q_f$.
Then for any $x\in\Fq$ and $m,n\in Q_f$, it holds that
$x\circ(m+n)=x(m+n)=xm+xn=x\circ m+x\circ n$, and
$x\circ(m\circ n)=x(m\circ n)=x\rho_n(m)=\rho_n(xm)=(x\circ m)\circ n$.
This implies $\Fq\subseteq K(Q_f)$.

Since $\Fqt\setminus L_f$ is not a subgroup of $(\Fqt,+)$, two elements
$m,n\in\Fqt\setminus L_f$ exist such that $w=m+n$ belongs to $L_f$.
Let $x\in Q_f\setminus\Fq$, and assume $x\circ w=xw$.
On the other hand, $x\circ w=h^{-1}f(hx)$ with $f(h)-wh=0$ for some $h\neq0$.
Hence
\[
  \frac{f(hx)}{hx}=w=\frac{f(h)}h
\]
that together with $\la h\ra_{\Fq}\neq\la hx\ra_{\Fq}$ contradicts the assumption that
$f(x)$ is scattered.
As a consequence,
\[
  x\circ(m+n)\neq x(m+n)=x\circ m+x\circ n
\]
and this implies $x\notin K(Q_f)$.
\end{proof}
%In the previous proof, $Q_f$ has been proved not to be a semifield.
%The quasifield is not even a nearfield, that is, the product is not associative.
%\begin{proposition}
%  The quasifield $(Q_f,+,\circ)$ is not a nearfield.
%\end{proposition}
%\begin{proof}
%First of all, there exist $m,n\in \Fqt\setminus L_f$ such that $mn\in L_f$, for otherwise
%$\Fqt^*\setminus L_f$ would be a subgroup of $(\Fqt^*,\cdot)$;
%since $\#(\Fqt^*\setminus L_f)=(q-2)(q^t-1)/(q-1)$ this is not possible.

%So, take $x\in\Fqt\setminus\Fq$, and $m,n\in \Fqt\setminus L_f$ such that $mn\in L_f$.
%Recall from the previous proof that $x\circ y=xy$ if and only if $y\notin L_f$.
%Therefore,
%\[
%  (x\circ m)\circ n=(xm)\circ n=xmn\neq x\circ(mn)=x\circ(m\circ n).
%\]
%\end{proof}

Let $\cB_f=\cB(Q_f)$ and $\ptra_f$ be the planar spread of $Q_f^2$ associated with $Q_f$
and the related translation plane, respectively.

\begin{remark}\label{r:generalizzabile}
The elements of $\cB_f$ distinct from $V_\infty$ are of two types:
\begin{enumerate}
\item
 if $m\in\Fqt\setminus L_f$, then $V_m=\la(1,m)\ra_{\Fqt}$;
\item
  if $m\in L_f$ and $f(h)-mh=0$, $h\neq0$, then
  \[
    V_m=\{(h^{-1}x,h^{-1}f(x))\colon x\in\Fqt\}=h^{-1}U_f.
  \]
\end{enumerate}
In particular, $U_f=V_{f(1)}$.
\end{remark}
Note that $\{V_m\colon m\in L_f\}$ and $\{\la(1,m)\ra_{\Fqt}\colon m\in L_f\}$
are hyper-reguli covering the same vector set.
Their union in the related projective space is a hypersurface \cite{LaShZa15}.

%Below the original definition of a linear set of pseudoregulus type in a projective line given
%in \cite{LuMaPoTr14} is rephrased into an equivalent one in order to obtain $0,1\notin L_f$.
\begin{definition}\cite[Definition 3.1]{LuMaPoTr14}\label{d:pseudo}
Let $s\in\mathbb N^*$ and $\omega\in\Fqt$ such that $(s,t)=1$, and $\N_{q^t/q}(\omega)\neq0,1$.
Any subset of $\PG(1,q^t)$ projectively equivalent to $L_{g_s}$ where $g_s(x)=\omega x^{q^s}$ is 
called an \emph{$\Fq$-linear set of pseudoregulus type}.

\begin{remark}
Note that the definition above has been rephrased in order to obtain $1\notin {L_{g_s}}$.
In fact, $1\in {L_{g_s}}$ would imply that $y\in\F_{q^t}^*$ exists such that 
$\omega y^{q^s-1}=1$, hence $\N_{q^t/q}(\omega)=1$, a contradiction.

The reason for such a choice is that we want the assumptions stated at the beginning of this section 
to be satisfied.
All linear sets of pseudoregulus type in $\PG(1,q^t)$ are projectively equivalent by definition.
In case $L_f$ is a linear set of pseudoregulus type and $0\in L_f$ or $1\in L_f$,
then~\eqref{e:quasicorpo} does not define a quasifield, although $\cB_f$ is still a planar spread.
\end{remark}

\end{definition}
The linear set $L_{g_s}$ is scattered for any $s\in\mathbb N^*$ such that $(s,t)=1$ and 
for any $\omega\in\Fqt^*$. 
Under such assumptions,
$L_{g_s}=\{m\in\Fqt\colon \N_{q^t/q}(m)=\N_{q^t/q}(\omega)\}$. 
If $m\in L_{g_s}$ and $g_s(h)-mh=0$, $h\neq0$, then $m=\omega h^{q^s-1}$, and
$h^{-1}g_s(hx)=x^{q^s}m$.
Therefore, the product in $Q_{g_s}$ is
\[
  x\circ m=\begin{cases} xm&\mbox{ if }\N_{q^t/q}(m)\neq \N_{q^t/q}(\omega),\\
x^{q^s}m&\mbox{ if }\N_{q^t/q}(m)=\N_{q^t/q}(\omega).\end{cases}
\]
If $\N_{q^t/q}(m)=\N_{q^t/q}(\omega)$, then $m=h^{q^s-1}\omega$ implies
\[
V_m=h^{-1}\{(x,x^{q^s}\omega)\colon x\in\Fqt\}.
\]
Therefore, $V_\omega=U_{g_s}$.
For $s=1$, $\cB_{g_s}$ is mapped by the projectivity of matrix 
$\begin{pmatrix}1&0\\ 0&\omega^{-1}\end{pmatrix}$ into the spread used in \cite{LuPo01}
in order to construct an Andr\'e translation plane
which is therefore isomorphic to $\ptra_{\omega x^q}$ (see also Theorem \ref{t:an}).

The plane $\ptra_{g_s}$ associated with the polynomial described in Definition
\ref{d:pseudo} is the Andr\'e 
$q$-plane  obtained by one $q^s$-Andr\'e net replacement.

%The condition $0,1\notin L_f$ for a scattered polynomial $f(x)$ has the sole purpose of defining a 
%quasifield. 
If there is no need for a quasifield framework, one can omit the condition $0,1\notin L_f$ 
and introduce, in view of Remark \ref{r:generalizzabile},
the more general definition of a  translation plane associated to a scattered polynomial
 as follows.

\begin{definition}
Let $f(x)\in\Fqt[x]$ be a scattered $\Fq$-linearized polynomial.
Let $\cB_f$ be the collection of the following subspaces of $\Fqt^2$:
\[
  \la(1,m)\ra_{\Fqt}\     (m\in\Fqt\setminus L_f), \qquad   hU_f   \  (h\in\Fqt^*),
 \qquad\{0\}\times\Fqt.
\]

The \emph{translation plane associated with $f$},   also denoted by  $\ptra_f$, 
is $\ptra_f=\ptra(\cB_f)$.
\end{definition}

%%{\color{green}{}}

\section{Equivalence of the translation planes associated with scattered polynomials}
\label{S.equiv}
In this section $f(x),f'(x)\in\Fqt[x]$ are two scattered $\Fq$-linearized polynomials.
The proof concerning the equivalence between the corresponding translation planes
will be based on the following restriction of the
Andr\'e's isomorphism theorem to the finite case.
\begin{theorem}\cite{An54}\cite[Theorem 1.18]{Kn95}\label{t:an}
  Let $\cB_1$ and $\cB_2$ be planar spreads of $\Fqt^2$, and $\ptra_i=\ptra(\cB_i)$,
  $i=1,2$.
  Assume moreover that $F$ is the kernel of $\ptra_1$.
  Let $\alpha:\Fqt^2\rightarrow\Fqt^2$ be bijective.
  Then $\alpha$ induces an isomorphism between $\ptra_1$ and $\ptra_2$ if and only if
  there exists an $F$-semilinear mapping $\lambda:\Fqt^2\rightarrow\Fqt^2$ and a vector
  $u\in\Fqt^2$ such that $\lambda(\cB_1)=\cB_2$ and $\alpha(v)=\lambda(v)+u$
  for all $v\in\Fqt^2$.
\end{theorem}

\begin{theorem}\label{t:main}
  Let $q>3$.
  A bijection $\alpha:\Fqt^2\rightarrow\Fqt^2$ is an isomorphism between $\ptra_f$ and $\ptra_{f'}$
  if and only if $\alpha(v)=\lambda(v)+u$ for some $u\in\Fqt^2$ and
  $\lambda$ is an $\Fqt$-semilinear bijective map satisfying $\lambda(U_f)=h'U_{f'}$ with 
  $h'\in\Fqt^*$. 
  Therefore, the planes $\ptra_f$ and $\ptra_{f'}$ are isomorphic if and only if $U_f$ and $U_{f'}$
  belong to the same orbit under the action of $\GaL(2,q^t)$.
\end{theorem}
\begin{proof}
Assume that $\ptra_f$ and $\ptra_{f'}$ are isomorphic.
By Theorem \ref{t:an}, a $\lambda\in\GaL(2t,q)$ exists such that $\lambda(\cB_f)=\cB_{f'}$.
There exists an element of $\GaL(2,q^t)$ mapping $U_f$ to $U_g$ where $g(x)$ is a 
scattered linearized polynomial such that $0,1\not\in L_g$.
Such semilinear map transforms $\cB_f$ into $\cB_g$. 
Then Theorem \ref{t:an} and Proposition \ref{p:quasicorpo} imply that the kernel of $\ptra_f$
is $\Fq$.
By Proposition \ref{p:quasicorpo}, $\lambda$ is an $\Fq$-semilinear map, with companion
automorphism $\sigma$, say.
%Let $\cB_1=\cB(Q_f)$ and $\cB_2=\cB(Q_{f'})$.
Both partial spreads $\cB_{f}\cap\cD$ and $\cB_{f'}\cap\cD$ have size $q^t+1-(q^t-1)/(q-1)$, 
and since $q>3$ there
are three distinct elements $X_1$, $X_2$, $X_3$ of 
$\cB_f\cap\cD$ which are mapped by $\lambda$ into elements 
of $\cB_{f'}\cap\cD$. 
%Choose $X_3\neq\la(0,1)\ra_{\Fqt}$.
Correspondingly, there are $v_i,w_i\in\Fqt^2$, $i=1,2,3$, such that $v_3=v_1+v_2$, and
\[
  \la v_i\ra_{\Fqt}\in\cB_f,\quad \lambda(\la v_i\ra_{\Fqt})=\la w_i\ra_{\Fqt},\quad i=1,2,3.
\]
%Let $w_3=(a,b)$, $a\neq0$.
For $i=1,2$ a map $\lambda_i:\,\Fqt\rightarrow\Fqt$ exists such that 
$\lambda(xv_i)=\lambda_i(x)w_i$ for $x\in\Fqt$.
Clearly $\lambda_i$ is an $\Fq$-semilinear map with companion automorphism $\sigma$.
As a consequence, for any $x\in\Fqt$,
\begin{equation}\label{e:coord0}
  \lambda(xv_3)=\lambda_1(x)w_1+\lambda_2(x)w_2\in\la w_3\ra_{\Fqt}.
\end{equation}

The vectors $w_i$, $i=1,2,3$ are pairwise $\Fqt$-linearly independent, hence a $c\in\Fqt^*$
exists such that $\la w_3\ra_{\Fqt}=\la w_1+c w_2\ra_{\Fqt}$.
This implies $\lambda_2(x)=c\lambda_1(x)$ for any $x$, and
\begin{equation}\label{e:coord}
  \lambda(\xi v_1+\eta v_2)=\lambda_1(\xi)w_1+c\lambda_1(\eta)w_2\quad\mbox{for any }\xi,\eta\in\Fqt.
\end{equation}

%\end{proof}
%\end{document}

Assume now that $n\in\Fqt^*$ and $z_n\in(\Fqt^2)^*$ satisfy 
\begin{equation}\label{e:topica}
\la v_1+nv_2\ra_{\Fqt}\in\cB_f,\qquad \lambda(\la v_1+nv_2\ra_{\Fqt})=\la z_n\ra_{\Fqt}.
\end{equation}
By \eqref{e:coord} and \eqref{e:topica}, for any  $x\in\Fqt$, $\lambda_1(x)$ and $\lambda_1(nx)$
are the coordinates of the vector $w(x)=\lambda(x(v_1+nv_2))$ in $\la z_n\ra_{\Fqt}$ with respect to 
the basis  $\{w_1,cw_2\}$.
Since $w(x)$ ranges in a one-dimensional subspace of $\Fqt^2$,
\[
  \begin{vmatrix}\lambda_1(x)&\lambda_1(nx)\\ \lambda_1(1)&\lambda_1(n)\end{vmatrix}=0
\]
for all $x\in\Fqt$.
Hence the $\Fq$-semilinear map $\rho(x)=\lambda_1(x)/\lambda_1(1)$ satisfies the condition
\begin{equation}\label{e:topica2}
\rho(nx)=\rho(n)\rho(x),\quad\mbox{for all }x\in\Fqt.
\end{equation}
%Since precisely $(q^t-1)/(q-1)$ elements of $\cB_f$ do not belong to $\cD$,
%there are at least 
%$q^t-(q^t-1)/(q-1)$ values of $n\in\Fqt$ such that \in\cB_f$;
%at most $(q^t-1)/(q-1)$ of such elements of $\cB_f$ are \textbf{not} mapped by $\lambda$ into 
%elements of $\cD$.
%So, since $q>3$ implies

The elements of $\cB_f$ which are also in the Desarguesian spread are precisely
\[ q^t+1-\frac{q^t-1}{q-1}. \]
All of them except $\la v_2\ra_{\Fqt}$ are of type $\la v_1+nv_2\ra_{\Fqt}$, $n\in\Fqt$.
Since $\lambda$ is a bijection between $\cB_f$ and $\cB_{f'}$ and
$\#(\cB_{f'}\setminus\cD)=(q^t-1)/(q-1)$, the number of elements of type
$\la v_1+nv_2\ra_{\Fqt}$ satisfying \eqref{e:topica} is at least
\[
  N=q^t-2\frac{q^t-1}{q-1},
\]
and $q>3$ implies $N>q^{t-1}$.
Then there exists an $\Fq$-basis $\{n_1,n_2,\ldots,n_t\}$ of $\Fqt$ such that
the condition \eqref{e:topica} is satisfied for any $n=n_i$, $i=1,2,\ldots,t$.
From \eqref{e:topica2}, for any $y,x\in\Fqt$, $y=\sum_{i=1}^ty_in_i$ ($y_i\in\Fq$, $i=1,2,\ldots,t$),
\[
\rho(yx)=\sum_iy_i^\sigma\rho(n_ix)=\sum_iy_i^\sigma\rho(n_i)\rho(x)=\rho(y)\rho(x).
\]
As a consequence, $\rho$ is an automorphism of the field $\Fqt$.
Condition \eqref{e:coord} yields
\[
  \lambda(\xi v_1+\eta v_2)=
  \lambda_1(1)\left(\rho(\xi)w_1+\rho(\eta)cw_2\right)\quad\mbox{for any }\xi,\eta\in\Fqt.
\]
So, $\lambda\in\GaL(2,q^t)$.
Then $\lambda(\cD)=\cD$, and since $\lambda(\cB_f)=\cB_{f'}$, one obtains
$\lambda(\cB_f\setminus\cD)=\cB_{f'}\setminus\cD$.
The elements of $\cB_f\setminus\cD$ are of type $hU_f$, $h\in\Fqt^*$, and the
elements of $\cB_{f'}\setminus\cD$ are of type $h'U_{f'}$, $h'\in\Fqt^*$.
Then an $h'\in\Fqt^*$ exists such that $\lambda(U_f)=h'U_{f'}$.
%Therefore, $U_f$ and $U_{f'}$ belong to the same orbit under the action of $\GaL(2,q^t)$.

Conversely, assume that $\lambda(U_f)=h'U_{f'}$ for some $h'\in\Fqt^*$ and 
$\lambda\in\GaL(2,q^t)$ with companion
automorphism $\sigma$.
This implies $\lambda(hU_f)=h^\sigma h' U_{f'}$ for any $h\in\Fqt^*$, hence
$\lambda(\cB_f\setminus\cD)=\cB_{f'}\setminus\cD$.
If $\la v\ra_{\Fqt}\in\cB_f$, then 
$\la v\ra_{\Fqt}\cap U_f=\{0\}$, and consequently
$\lambda(\la v\ra_{\Fqt})\cap U_{f'}=\{0\}$.
Since the elements of $\cD$ which intersect trivially $U_{f'}$ belong to $\cB_{f'}$,
it follows  $\lambda(\la v\ra_{\Fqt})\in\cB_{f'}$.
As a result, the assumptions of Theorem \ref{t:an} are satisfied and $\alpha$ is an
isomorphism between $\ptra_f$ and $\ptra_{f'}$.
\end{proof}
\begin{corollary}\label{c:nt} Assume $q>3$.
If $\alpha$ is an automorphism of $\ptra_f$, and $\lambda(v)=\alpha(v)-\alpha(0)$
for any $v\in\Fqt^2$, then $\lambda$ maps elements of $\cB_f\cap \cD$ into elements of 
$\cB_f\cap \cD$, and elements of $\cB_f\setminus\cD$ into elements of $\cB_f\setminus\cD$.
\end{corollary}

\begin{corollary} Assume $q>3$.
The plane $\ptra_f$ is neither a semifield plane nor a nearfield plane.
\end{corollary}
\begin{proof}
By Theorem \ref{t:main} it may be assumed again that $0,1\notin L_f$.
If $\ptra=\ptra(\cB)$, $\cB$ a planar spread of $\mathbb V$, is a semifield plane, 
then there is
an $S\in\cB$ and a group contained in $\GaL(\mathbb V,K(\ptra))$ preserving $\cB$ and acting
transitively on $\cB\setminus\{S\}$ \cite{Os77}. 
Similarly, if $\ptra$ is a nearfield plane, 
then there are distinct
$S,T\in\cB$ and a group contained in $\GaL(\mathbb V,K(\ptra))$ preserving $\cB$ and acting
transitively on $\cB\setminus\{S,T\}$ \cite{Os77}. 
%By Corollary \ref{c:nt}, if $m\in L_f$ and $n\in\Fqt\setminus L_f$, then no
%$\lambda\in\GaL(\Fqt^2,\Fq)$ preserving $\cB_f$ maps $V_m$, that does not belong to the
%Desarguesian spread, into $V_n=\la n\ra_{\Fqt}$.
As a consequence, if $\cA_f$ is a semifield or a nearfield plane, a 
$\lambda\in\GaL(\Fqt^2,\Fq)$ preserving $\cB_f$ exists
mapping some element of $\cD$ into an element of $\cB_f\setminus\cD$.
This contradicts Corollary~\ref{c:nt}.
\end{proof}

\begin{definition}\cite{CsMaPo18}
Let $L_U$ be an $\Fq$-linear set of $\PG(1, q^t)$ of maximum rank with
maximum field of linearity $\F_q$; that is, $L_U$ is not an $\F_{q^r}$-linear set for any $r>1$.
Then $L_U$ is said of \emph{$\GaL$-class 
$c=c_\Gamma(L)$} if $c$ is the largest integer
such that there exist $\Fq$-subspaces $U_1, U_2, \ldots, U_c$ of $\Fqt^2$ with 
$L_{U_i} = L_U$ for $i\in\{1, 2, \ldots, c\}$
and there is no $\kappa\in\GaL(2, q^t )$ such that $U_i = \kappa(U_j)$ 
for each $i\neq j$, $i, j\in\{1, 2, \ldots, c\}$.
\end{definition}
\begin{corollary}\label{c:class}
Let $q>3$.
Any scattered $\Fq$-linear set $L_U$ of maximum rank in $\PG(1,q^t)$ gives rise to 
$c_\Gamma(L)$
pairwise nonisomorphic translation planes.
\end{corollary}
As a consequence of the results in \cite{CsZa16}, the linear set of pseudoregulus type in
$\PG(1,q^t)$ is of $\GaL$-class $\phi(t)/2$, where $\phi(t)$ is the totient function.
The sets $U_{x^{q^s}}$ and $U_{x^{q^{s'}}}$ are in a common orbit under the
action of $\GaL(2,q^t)$ if and only if $s\equiv \pm s'\Mod t$.
This implies the following result.
\begin{theorem}\label{th:pseudoreg}
Let $q>3$.
If $f(x)= x^{q^s}$, $f'(x)= x^{q^{s'}}$ are polynomials in $\Fqt[x]$,
and $(s,t)=1=(s',t)$, then $\ptra_f$ and $\ptra_{f'}$
are isomorphic if and only if $s\equiv \pm s'\Mod t$.
\end{theorem}

\begin{remark}
The main argument in the proof of Theorem \ref{t:main} can be generalized to obtain
the following result.
\begin{theorem}
Let $\cB_1$ and $\cB_2$ be planar spreads of $V(\Fqt^2,\Fq)$,
obtained with $\ell_1$ and $\ell_2$ hyper-reguli replacements from the Desarguesian spread $\cD$,
respectively.
Assume $\ell_1+\ell_2\le q-2$.
If $\lambda\in\GaL(2t,q)$ satisfies $\lambda(\cB_1)=\cB_2$, then $\lambda$ is $\Fqt$-semilinear.
\end{theorem}
\end{remark}

\section{Translation planes associated with the Lunardon-Polverino polynomials}\label{s:LP}

The \emph{Lunardon-Polverino polynomial}, or \emph{LP-polynomial}, of indices $b\in\Fqt^*$ and
$s\in\mathbb N$ is
\[
P_{b,s}=x^{q^s}+x^{q^{t-s}}b.
\]
Such a polynomial is scattered if \cite{LuPo01} and only if \cite{Za19} $\N_{q^t/q}(b)\neq1$
and $(s,t)=1$.
In this case, for $t>3$ the related scattered linear set is not of pseudoregulus type 
\cite{LuPo01}.
\begin{theorem}\cite[Theorem 4.1]{JhJo08}
Consider the following set of subspaces of $V(\Fqt^2,\Fq)$:
\begin{equation}\label{e:fhr}
\mathcal{H}_{b,s}= \left \{  V_{b,d,s} \colon d\in\Fqt^*\right\},
\end{equation}
where
\begin{equation}\label{e:fhrsub}
  V_{b,d,s}=\left\{(x,x^{q^s}d^{1-q^s}+x^{q^{t-s}}d^{1-q^{t-s}}b)\colon x\in\Fqt \right\}
\end{equation}
and  $b^{(q^t-1)/(q^{(s,t)}-1)}$ is not equal to 1.

Then 
\begin{enumerate}[(1)]
\item
this set is a replacement set for a hyper-regulus of $V(\Fqt^2,\mathbb F_{q^{(s,t)}})$;
\item
$V_{b,d,s}$, $d \in \F_{q^{(s,t)}}$, is an $\mathbb F_{q^{(s,t)}}$-subspace and lies over the Desarguesian spread
in the sense that the subspace intersects exactly $(q^t-1)/(q^{(s,t)}-1)$ components in
1-dimensional $\mathbb F_{q^{(s,t)}}$-subspaces;
\item
if $t/(s,t)>3$, the hyper-regulus of components that this subspace lies over is not an Andr\'e
hyper-regulus.
\end{enumerate}
\end{theorem}

\begin{definition}\cite[\textit{Definition 4.2 \& Corollary 4.3}]{JhJo08}
Any hyper-regulus contained in the Desarguesian spread of $V(\Fqt^2,\mathbb F_{q^{(s,t)}})$
that has a replacement set of the form \eqref{e:fhr} shall be called \emph{fundamental hyper-regulus}.
\end{definition}

When $(s,t)=1$,  the hyper-regulus associated with
a scattered LP-polynomial turns out to be a fundamental hyper-regulus.
As a matter of fact,
\[
V_{b,d,s}=d\{(x,x^{q^s}+bx^{q^{t-s}})\colon x\in\Fqt\}=dU_{P_{b,s}}.
\]

As a consequence, the results in \cite{JhJo08} in particular describe the planes
$\ptra_{P_{b,s}}$ with $\N_{q^t/q}(b)\neq1$ and $(s,t)=1$.
In the next theorem:
\begin{itemize}
\item [-] an \emph{affine central collineation} is a collineation of the affine plane whose projective
extension is a central collineation, and whose axis is an affine line;
\item [-] the \emph{kernel homology group of the associated Desarguesian plane} contains all
collineations of type $(x,y)\mapsto(ax,ay)$ with $a\in\Fqt^*$; note that such maps are collineations
but not central collineations of $\ptra_{P_{b,s}}$;
\item [-] a \emph{coaxis} of an affine central collineation is any affine line through the 
center;
\item [-] given two groups $G$ and $H$ of affine homologies and two lines
$\ell$, $m$, if $\ell$ is axis of any element in $G$ and coaxis of any element in $H$,
and furthermore $m$ is axis of any element in $H$ and coaxis of any element in $G$, then
$G$ and $H$ are called \emph{symmetric affine homology groups};
\item [-] an \emph{elation} of an affine translation plane is intended to have proper axis.
\end{itemize}

\begin{theorem}\cite[Theorem 6.1]{JhJo08}\label{t:jj}
Assume that $q > 3$, $\N_{q^t/q}(b)\neq0,1$ and $(s,t)=1$. Then the following hold:
\begin{enumerate}[(1)]
\item
If $t>3$ is odd, then the plane $\ptra_{P_{b,s}}$ admits no affine central collineation
group and the full collineation group in $\GL(2, q^t )$ has order $(q^t - 1)$ and is the
kernel homology group of the associated Desarguesian plane.

\item
If $t>3$ is even,  then the plane $\ptra_{P_{b,s}}$ admits symmetric affine homology
groups of order $q + 1$ but admits no elation group. 
The full collineation group
in $\GL(2, q^t )$ has order $(q + 1)(q^t - 1)$, and is the direct product of the kernel
homology group of order $(q^t - 1)$ by a homology group of order $q + 1$.
\end{enumerate}
\end{theorem}
\begin{corollary}\cite[Corollary 6.2]{JhJo08}. 
The translation plane $\ptra_{P_{b,s}}$, $t>3$,  $q > 3$, $\N_{q^t/q}(b)\neq0,1$ and $(s,t)=1$, has
kernel $\Fq$ and cannot be a generalized Andr\'e or an Andr\'e plane.
\end{corollary}

In \cite[Section 2]{LoMaTrZh21}, the authors investigate the $\GaL$-equivalence of the set
of type $U_{P_{b,s}}$.
Rephrasing Proposition 2.3 there,
the result reads as follows:
\begin{proposition}\cite{LoMaTrZh21}
  Assume $s,s'<t/2$, and $b,b'\in\Fqt$ with $\N_{q^t/q}(b)\neq1\neq\N_{q^t/q}(b')$.
  Then the subspaces $U_{P_{b,s}}$ and $U_{P_{b',s'}}$ are in the same orbit under the
  action of $\GaL(2,q^t)$ if and only if $s=s'$, and there exist $z\in\Fqt$ and an automorphism 
  $\sigma$ of $\Fqt$ 
  such that
  \begin{equation}\label{e:ejj}
    b'=b^\sigma z^{q^{2s}-1}.
  \end{equation}
  As a consequence, if $q=p^e$, under the action of $\GaL(2,q^t)$ there are at least
  \begin{equation}\label{e:atleast} N_{q,t}=
    \begin{cases}    
    \frac{q-2}{e}\,\frac{\phi(t)}2&\mbox{ for odd }t\\
    \frac{q^2-1-(q+1)}{2e}\,\frac{\phi(t)}2&\mbox{for even }t
    \end{cases}
  \end{equation}
  orbits of subspaces of type $U_{P_{b,s}}$.
\end{proposition}
The subspaces $U_{P_{b,s}}$ and $U_{P_{b^{-1},t-s}}$ are in the same orbit under 
the action of $\GL(2,q^t)$.
Since for $q>3$ two planes $\ptra_f$ and $\ptra_{f'}$ are isomorphic if and only if
$U_f$ and $U_{f'}$ are in the same orbit under the action of $\GaL(2,q^t)$, one obtains
by \eqref{e:ejj}:
\begin{theorem}\label{t:nesatto}
The number of pairwise non-isomorphic planes of type $\ptra_{P_{b,s}}$, $t>3$,  $q > 3$,
equals
\begin{enumerate}[(1)]
\item the number of the orbits in $\Fq\setminus\{0,1\}$ under the action of the Galois group
$\Gal(\Fq/\Fp)$ over the prime subfield of $\Fq$, if $t$ is odd, or
\item the number of the orbits in $\F_{q^2}\setminus\{x^{q-1}\colon x\in\F_{q^2}\}$
under the action of $\Gal(\F_{q^2}/\Fp)$, if $t$ is even.
\end{enumerate}
\end{theorem}
\begin{corollary}\label{c:disaccordo}
  For $q>3$ and $t>3$ there are at least $N_{q,t}$  mutually non-isomorphic
  translation planes of order $q^t$ and kernel $\Fq$ that may be obtained from a  Desarguesian spread  
  by replacement of a hyper-regulus of the type $\mathcal{H}_{b,s}$ with $(s,t)=1$.
\end{corollary}

\begin{remark} Corollary \ref{c:disaccordo} does not agree with 
\cite[Theorem 5.5 \& Corollary 5.6]{JhJo08}.
It is opinion of the authors of this paper that the argument at \cite[p.96, line 11]{JhJo08}
is incorrect.
As a counterexample, let $t>3$ be odd, $q=p^e$, $p$ a prime, $e>1$, $0<s<t$ such that $(s,t)=1$;
furthermore, take $t$ and $e$ coprime with $q-2$.
In \cite{JhJo08} it is asserted that there are at least $(q-2)/(te,q-2)=q-2$ non-isomorphic
planes of type $\ptra_{P_{b,s}}$.
See also at p.97, line 5 from the bottom, where the same assertion is involved.
However, by Theorem \ref{t:nesatto},
the number of such planes is less than $q-2$.

\end{remark}

\section*{Acknowledgment}
The authors are grateful to Guglielmo Lunardon for sowing the seeds of this work, 
and for his valuable suggestions and comments.

\noindent
Valentina Casarino, Giovanni Longobardi and Corrado Zanella\\
Dipartimento di Tecnica e Gestione dei Sistemi Industriali\\
Universit\`a degli Studi di Padova\\
Stradella S. Nicola, 3\\
36100 Vicenza VI - Italy\\
\emph{\{valentina.casarino,giovanni.longobardi,corrado.zanella\}@unipd.it}

\end{document}